\documentclass[11pt,a4paper]{amsart}

\usepackage[T1]{fontenc}
\usepackage{lmodern}
\usepackage{microtype}
\usepackage{amsmath,amssymb,mathtools}
\usepackage[hidelinks]{hyperref}

\hypersetup{
  pdftitle={Polynomial Compressibility and Forbidden Oriented Forests},
  pdfauthor={Zhenhua Lyu},
  pdfsubject={Homomorphism compressibility of acyclic oriented graphs},
  pdfkeywords={oriented graph, graph homomorphism, compressibility,
    forbidden oriented forest, relative oriented clique}
}

\allowdisplaybreaks

\newtheorem{theorem}{Theorem}[section]
\newtheorem{lemma}[theorem]{Lemma}
\newtheorem{corollary}[theorem]{Corollary}
\theoremstyle{remark}
\newtheorem{remark}[theorem]{Remark}

\newcommand{\tourn}[1]{\vec{T}_{#1}}
\newcommand{\dpath}[1]{\vec{P}_{#1}}
\newcommand{\dcycle}[1]{\vec{C}_{#1}}
\newcommand{\Fturn}{F_{\mathrm{turn}}}
\newcommand{\Falt}{F_{\mathrm{alt}}}
\newcommand{\Fdir}{F_{\mathrm{dir}}}
\newcommand{\Nplus}{N^{+}}
\newcommand{\Nminus}{N^{-}}
\newcommand{\dplus}{d^{+}}
\newcommand{\dminus}{d^{-}}
\newcommand{\ttr}{r_{\mathrm{tr}}}
\newcommand{\girth}{\operatorname{girth}}
\newcommand{\omro}{\omega_{\mathrm{ro}}}
\newcommand{\omao}{\omega_{\mathrm{ao}}}

\title[Polynomial compressibility and oriented forests]
  {Polynomial Compressibility and Forbidden Oriented Forests}
\author{Zhenhua Lyu}
\address{School of Science, Shenyang Aerospace University,
  Shenyang 110136, China}
\email{lyuzhh@outlook.com}


\begin{document}

\begin{abstract}
For a nonempty acyclic oriented graph $H$, let $p(H)$ be the order of a longest
directed path and let $\tau(H)$ be the least positive integer $n$ such that $H$
admits a homomorphism to every tournament of order $n$.  For all $p\ge3$
and $g\ge1$, we construct a connected acyclic oriented graph $H$ with
underlying girth greater than $g$, absolute and relative oriented clique
numbers equal to three, and
\[
  p(H)=p,\qquad \tau(H)=r_{\mathrm{tr}}(p),
\]
where $r_{\mathrm{tr}}(p)=2^{\Theta(p)}$ is the tournament Ramsey number
for a transitive $p$-vertex tournament.  This disproves the conjectured
polynomial bounds under bounded absolute or relative oriented clique
number.  It also shows that a forbidden graph can yield a polynomially
$\tau$-bounded class only if its underlying graph is a forest.  For fixed
$g$, the least order of these examples is bounded by a polynomial in $p$.  A separate
construction gives maximum in- and outdegree $O(p^2)$, uniformly in $g$.
For $p=4$, the least order is $2^{\Theta(g)}$.
We also establish polynomial $\tau$-boundedness for every orientation of
the two four-vertex trees.  The pure-claw case follows from the known
$O(p^4)$ bound.  We obtain the bound $2p-2$ for mixed claws and one-turn
orientations of $P_4$ when $p\ge2$, and bounds $4$ and $3p-2$ for the
directed and alternating orientations of $P_4$, respectively.  In the
alternating case, $\tau(H)=p(H)$ when the underlying graph is triangle-free.
\end{abstract}

\maketitle

\section{Introduction}

All graphs and digraphs in this paper are finite.
An \emph{oriented graph} is a loopless digraph with no pair of opposite
arcs, and a \emph{tournament} is an orientation of a complete graph.  A
homomorphism $f\colon H\to T$ is a vertex map that preserves every arc.

The \emph{compressibility} $\tau(H)$ is the least positive integer $n$ such
that $H\to T$ for every tournament $T$ of order $n$; we put $\tau(H)=\infty$
if no such integer exists.  This parameter was
introduced by Valadkhan~\cite{Valadkhan2009} in connection with oriented
Tur\'{a}n problems.  In particular, when $2\le\tau(H)<\infty$, it determines the
asymptotic oriented Tur\'{a}n number by
\[
  \operatorname{ex}_{\mathrm{o}}(n,H)
  =\left(1-\frac{1}{\tau(H)-1}\right)\binom n2+o(n^2).
\]
The parameter is finite exactly for acyclic oriented graphs.  For a nonempty
acyclic graph $H$, let $p(H)$ be the maximum number of vertices on a directed path.
The longest-path rank map gives $H\to\tourn{p(H)}$, whereas the directed
path of order $p(H)$ cannot map to $\tourn{p(H)-1}$.  Consequently,
\begin{equation}\label{eq:universal-lower}
  \tau(H)\ge p(H).
\end{equation}

Let
\[
  \ttr(p):=\min\{n:\text{every tournament of order $n$ contains }
  \tourn p\}.
\]
Since every acyclic $H$ with $p(H)=p$ maps to $\tourn p$, we have the
universal upper bound
\begin{equation}\label{eq:universal-upper}
  \tau(H)\le\ttr(p(H)).
\end{equation}

Gerbner, Hu, and Sun~\cite{GerbnerHuSun2026} determined the exact oriented
Tur\'{a}n number for every forbidden oriented graph with at most three arcs,
provided the host order is sufficiently large.  Here we fix an oriented
graph $F$ and ask whether $\tau(H)$ is polynomially bounded in $p(H)$
among $F$-free acyclic oriented graphs.
Wang and Lu~\cite{WangLu2025} determined the compressibility of directed
paths with separated additional arcs.  Their triangle-chain examples will
also be used below.

Two vertices are \emph{linked} if they are the endpoints of a directed path
of length at most two, in either direction.  A pairwise linked set is a
\emph{relative oriented clique}; its maximum order is $\omro(H)$.  An
\emph{oriented clique} is an oriented graph in which every two vertices are
linked within the graph, and $\omao(H)$ is the maximum order of an oriented
clique occurring as a subgraph of $H$.  Thus $\omao(H)\le\omro(H)$; see
Sopena~\cite{Sopena2016} for background.

For $k\ge3$, Grzesik, Jaworska, Kielak, Novik, and
\'{S}lusarczyk~\cite{GrzesikEtAl2024} considered
\[
\begin{aligned}
  \mathcal A_k&=\{H:H\text{ is acyclic and }\omao(H)\le k\},\\
  \mathcal R_k&=\{H:H\text{ is acyclic and }\omro(H)\le k\},
\end{aligned}
\]
and conjectured that both classes are polynomially $\tau$-bounded: there
are constants $c,d>0$, depending only on $k$, such that
$\tau(H)\le c\,p(H)^d$ for every nonempty graph in the class.  The following theorem
disproves this conjecture for every $k\ge3$.

\begin{theorem}\label{thm:high-girth}
For every pair of integers $p\ge3$ and $g\ge1$, there is a connected
acyclic oriented graph $H_{p,g}$ satisfying
\[
  \girth(U(H_{p,g}))>g,
  \qquad
  \omao(H_{p,g})=\omro(H_{p,g})=3,
\]
and
\[
  p(H_{p,g})=p,
  \qquad
  \tau(H_{p,g})=\ttr(p).
\]
\end{theorem}

Together with~\eqref{eq:universal-upper}, Theorem~\ref{thm:high-girth} gives
\begin{equation}\label{eq:exact-maximum}
\begin{aligned}
  \max_{\substack{H\in\mathcal R_3\\p(H)=p}}\tau(H)
  &=\max_{\substack{H\in\mathcal A_3\\p(H)=p}}\tau(H)\\
  &=\max_{\substack{H\text{ acyclic}\\p(H)=p}}\tau(H)
   =\ttr(p)
\end{aligned}
\end{equation}
for every $p\ge3$.  The same identity holds if all three maxima are
restricted to connected graphs whose underlying girth exceeds any prescribed
$g$.

The classical tournament Ramsey bounds give
\begin{equation}\label{eq:classical-ramsey}
  2^{p/2-O(1)}\le\ttr(p)\le2^{p-1};
\end{equation}
see Erd\H{o}s and Moser~\cite[Section~1]{ErdosMoser1964} and
Manoussakis and Tuza~\cite{ManoussakisTuza2001}.
Thus the largest compressibility at path order $p$ is exponential even
in $\mathcal R_3$.  The clique value three is best possible for $p\ge3$,
since three consecutive vertices of a directed path form an oriented clique.

Let $m_g(p)$ be the least order of a graph in Theorem~\ref{thm:high-girth},
and let $d_g(p)$ be the least possible value of
$\max\{\Delta^+(H),\Delta^-(H)\}$ among these graphs.  For
$h=\max\{g,12\}$, Theorem~\ref{thm:quantitative-lifts} gives
\[
  m_g(p)\le C_h p^{3h-1},
  \qquad
  d_g(p)\le96p(p-1),
\]
where $C_h$ depends only on $h$.  The two bounds come from separate
constructions, and the degree bound is uniform in $g$.  An entropy estimate
improves the degree bound from $O(p^3)$ to $O(p^2)$ compared with a direct
union bound over all maps.  For the dependence on $g$,
Lemma~\ref{lem:girth-order-lower} gives
\[
  m_g(4)=2^{\Theta(g)}.
\]

\begin{corollary}\label{cor:clique-conjecture}
For every $k\ge3$, neither $\mathcal A_k$ nor $\mathcal R_k$ is
polynomially $\tau$-bounded.  For every prescribed $g$, the conclusion
already holds among connected graphs of underlying girth greater than $g$.
Moreover, if $F$ is an acyclic oriented graph and $U(F)$ contains a cycle,
then the class of $F$-free acyclic oriented graphs is not polynomially
$\tau$-bounded.
\end{corollary}

The forbidden-subgraph problem posed in~\cite{GrzesikEtAl2024} therefore
reduces to oriented forests.  We prove polynomial $\tau$-boundedness for
all orientations of the two four-vertex trees.  For $a+b=3$, let $F_{a,b}$
be the orientation of $K_{1,3}$ whose center has indegree $a$ and outdegree $b$.  Up to reversing all
arcs and reversing the order of the underlying path, the three orientations
of $P_4$ are
\[
\begin{aligned}
  \Fdir&:\quad a\to b\to c\to d,\\
  \Fturn&:\quad a\to b\to c\leftarrow d,\\
  \Falt&:\quad a\to b\leftarrow c\to d.
\end{aligned}
\]
Here and throughout, $F$-free means that no subgraph is isomorphic to $F$;
the copy need not be induced.

\begin{theorem}\label{thm:four-vertex-trees}
Let $H$ be a nonempty acyclic oriented graph and put $p=p(H)$.
\begin{enumerate}
  \item If $H$ is $F_{0,3}$-free or $F_{3,0}$-free, then
  $\tau(H)=O(p^4)$.  If $H$ is $F_{1,2}$-free or $F_{2,1}$-free and
  $p\ge2$, then $\tau(H)\le2p-2$.
  \item If $H$ is $\Fdir$-free, then $p\le3$ and $\tau(H)\le4$.
  If $H$ is $\Fturn$-free and $p\ge2$, then $\tau(H)\le2p-2$.
  If $H$ is $\Falt$-free, then $\tau(H)\le3p-2$; if in addition
  $U(H)$ is triangle-free, then $\tau(H)=p$.
\end{enumerate}
The same conclusions hold for the orientations obtained by reversing all
arcs.
\end{theorem}

The pure-claw assertion follows from the bounded-outdegree result
of~\cite[Section~3]{GrzesikEtAl2024}.
For the alternating path, Wang and Lu's triangle chains
$H_q$~\cite[Section~4]{WangLu2025} satisfy
\[
  p(H_q)=2q-1,
  \qquad
  \tau(H_q)=3q-2=\frac{3p(H_q)-1}{2}.
\]
These graphs are $\Falt$-free, as verified in Section~\ref{sec:sharpness},
so a general linear bound for this class must have coefficient at least $3/2$.

The high-girth construction uses Kun's sparse incomparability
theorem~\cite{Kun2013}.  Kayll and Morris~\cite{KayllMorris2023} proved a
related result for oriented graphs with large directed girth.
The four-vertex-tree proofs are elementary.
For a mixed claw, we map $H$ to a fixed four-vertex tournament.  For the
one-turn path, deleting the sources leaves an out-forest.  In the
alternating case, a local degree restriction gives a line-digraph
representation when the underlying graph is triangle-free.  Deleting the
middle vertices of transitive triangles then reduces the general case
to this one.  Section~\ref{sec:sharpness} discusses sharpness and further
questions.

\section{Preliminaries}\label{sec:preliminaries}

For a vertex $v$, let $\Nplus(v)$ and $\Nminus(v)$ be its out- and
in-neighborhoods, and let $\dplus(v)$ and $\dminus(v)$ be the corresponding
degrees.  The underlying simple graph of an oriented graph $H$ is denoted by
$U(H)$.  We write $\tourn k$ for the transitive tournament of order $k$,
with vertex set $[k]$ and arcs $ij$ when $i<j$, and $\dpath k$ and
$\dcycle k$ for the directed path and directed cycle of order $k$.

Assigning to each vertex $v$ of an acyclic oriented graph the order of a
longest directed path ending at $v$ gives the rank homomorphism
\begin{equation}\label{eq:rank-map}
  H\longrightarrow\tourn{p(H)}.
\end{equation}
The defining property of $\tau$ is monotone in the tournament order: once
$H$ maps to every tournament of order $m$, it maps to every larger
tournament by first choosing an $m$-vertex subtournament.  Reversing all
arcs preserves both $p(H)$ and $\tau(H)$.

By R\'edei's theorem~\cite{Redei1934}, every nonempty tournament has
a directed Hamiltonian path.

For disjoint oriented graphs $A$ and $B$, let $A\Rightarrow B$ be their
union together with all arcs directed from $A$ to $B$.

\begin{lemma}\label{lem:relative-clique}
If an oriented graph $H$ has $\girth(U(H))>12$, then $\omro(H)\le3$.
\end{lemma}

\begin{proof}
Suppose that $x_1,x_2,x_3,x_4$ form a relative oriented clique.  For each of
the six pairs, choose a directed path of length at most two linking its
endpoints, and let $F$ be the union of the underlying edges of these paths.
The graph $F$ is connected and has at most twelve edges.  Since
$\girth(U(H))>12$, it is a tree.

Every chosen linking path is now the unique path between its distinguished
endpoints, so the four $x_i$ are pairwise at distance at most two in $F$.
Their minimal connecting subtree $F_0$ has only distinguished leaves.  Its
diameter is at most two and hence $F_0$ is a star.  Its center is either one
of the four distinguished vertices or a fifth vertex; in either case at
least three distinguished leaves are incident with the center.  Among their
three incident edges, two point in the same direction relative to the
center.  The unique path between the corresponding leaves is not directed,
contradicting the choice of its linking path.
\end{proof}

We use the following form of Kun's sparse incomparability
theorem~\cite[Theorem~1]{Kun2013}; see also the digraph formulation
in~\cite[Theorem~7]{GuzmanProMartin2025}.

\begin{theorem}\label{thm:sil}
Let $S$ be an oriented graph, and let $t\ge2$ and $g\ge1$ be integers.
There is an oriented graph $G$ such that
\[
  G\to S,\qquad \girth(U(G))>g,
\]
and, for every oriented graph $X$ with $|V(X)|<t$,
\[
  G\to X\quad\Longleftrightarrow\quad S\to X.
\]
\end{theorem}

Kayll and Morris~\cite[Theorem~1.1]{KayllMorris2023} proved a related
result for oriented graphs, stated in terms of directed girth.

\section{Exact high-girth realizations}\label{sec:high-girth}

\begin{proof}[Proof of Theorem~\ref{thm:high-girth}]
Fix $p\ge3$ and $g\ge1$, and put $r=\ttr(p)$ and $h=\max\{g,12\}$.
Choose a $\tourn p$-free tournament $Q$ of order $r-1$.
It contains $\tourn{p-1}$: otherwise adding a dominating vertex would
give a $\tourn p$-free tournament of order $r$.

A homomorphism between tournaments is injective, so
$\tourn p\nrightarrow Q$.  Applying Theorem~\ref{thm:sil} with
$S=\tourn p$, $t=r$, and girth threshold $h$ gives an oriented graph
$G$ such that
\[
  G\to\tourn p,\qquad G\nrightarrow Q,\qquad \girth(U(G))>h.
\]
The first property makes $G$ acyclic.  The second gives a weak component
$C$ with $C\nrightarrow Q$.  Since $C\to\tourn p$, we have $p(C)\le p$.
If $p(C)\le p-1$, the rank map and $\tourn{p-1}\to Q$ would give
$C\to Q$.  Thus $p(C)=p$.  Moreover, the obstruction $C\nrightarrow Q$
and~\eqref{eq:universal-upper} give
\[
  r\le\tau(C)\le\ttr(p)=r.
\]
The component $C$ has underlying girth greater than $h$, so
Lemma~\ref{lem:relative-clique} gives $\omro(C)\le3$.
Three consecutive vertices of a longest directed path form an oriented
clique, hence $\omao(C)=\omro(C)=3$.  Taking $H_{p,g}=C$ proves the theorem.
\end{proof}

\begin{remark}\label{rem:core}
The graph $H_{p,g}$ may be chosen to be a core.  Take the image $K$ of a
retraction onto a core.  Homomorphic equivalence preserves $p$ and $\tau$;
as a subgraph, $K$ preserves the girth and clique upper bounds, while
$p(K)\ge3$ forces equality in both clique bounds.  A homomorphic image of a
connected graph is connected.
\end{remark}

\begin{proof}[Proof of Corollary~\ref{cor:clique-conjecture}]
For every fixed $g$, Theorem~\ref{thm:high-girth}
and~\eqref{eq:classical-ramsey} give graphs $H_{p,g}\in\mathcal R_3$
with $p(H_{p,g})=p$ and $\tau(H_{p,g})=2^{\Theta(p)}$.
Since $\mathcal R_3\subseteq\mathcal R_k\subseteq\mathcal A_k$ for
$k\ge3$, neither class is polynomially $\tau$-bounded.
If $U(F)$ contains a cycle, take $g=|V(F)|$.  The same graphs are
$F$-free, since a copy of $F$ would contain a cycle of length at most $g$.
\end{proof}

\begin{remark}
The same argument applies to a finite nonempty family $\mathcal F$ if every
$U(F)$, $F\in\mathcal F$, contains a cycle: take
$g=\max_{F\in\mathcal F}|V(F)|$.
\end{remark}

The in- and outdegrees of these examples must grow with $p$.

\begin{lemma}\label{lem:degree-lower}
Every acyclic oriented graph $H$ with $p(H)=p\ge3$ and
$\tau(H)=\ttr(p)$ satisfies
\[
  \Delta^+(H)
  =\Omega\!\left(\frac{p}{(\log p)^2}\right),
  \qquad
  \Delta^-(H)
  =\Omega\!\left(\frac{p}{(\log p)^2}\right).
\]
\end{lemma}

\begin{proof}
There is an absolute constant $c>0$ such that every acyclic oriented graph
$H$ with $p(H)=p$ and $\Delta^+(H)\le k$ satisfies
\begin{equation}\label{eq:quasipoly-degree}
  \tau(H)\le(kp)^{ck\log p};
\end{equation}
this bound is recorded in~\cite[Section~3]{GrzesikEtAl2024} and is obtained
by modifying the proof of Fox, He, and
Wigderson~\cite[Theorem~1.4]{FoxHeWigderson2024}.  Put
$k=\Delta^+(H)$.  Equation~\eqref{eq:classical-ramsey} and the hypothesis on $\tau(H)$
give, for all sufficiently large $p$,
\[
  2^{p/3}\le\tau(H)\le(kp)^{ck\log p}.
\]
If $k\le p$, take logarithms and use $\log(kp)\le2\log p$ to get
$k=\Omega(p/(\log p)^2)$; if $k>p$, the conclusion is immediate.  Apply
the same argument to the arc reversal of $H$ to obtain the indegree
bound.
\end{proof}

\section{Quantitative bounds}\label{sec:quantitative-lifts}

For $p\ge3$ and $g\ge1$, let $\mathcal H_g(p)$ be the family of connected
acyclic oriented graphs $H$ such that
\[
  \girth(U(H))>g,\qquad
  \omao(H)=\omro(H)=3,
\]
and
\[
  p(H)=p,\qquad \tau(H)=\ttr(p).
\]
This family is nonempty by Theorem~\ref{thm:high-girth}.  Define
\[
  m_g(p):=\min_{H\in\mathcal H_g(p)}|V(H)|,
  \qquad
  d_g(p):=\min_{H\in\mathcal H_g(p)}
  \max\{\Delta^+(H),\Delta^-(H)\}.
\]
Lemma~\ref{lem:degree-lower} gives, uniformly in $g$,
\[
  d_g(p)=\Omega\!\left(\frac{p}{(\log p)^2}\right).
\]

\begin{theorem}\label{thm:quantitative-lifts}
Put $h=\max\{g,12\}$.  There is a constant $C_h$, depending only on
$h$, such that, for every $p\ge3$,
\[
  m_g(p)\le C_h p^{3h-1},
  \qquad
  d_g(p)\le96p(p-1)<96p^2.
\]
The two estimates are not asserted to hold for a single graph.
\end{theorem}

The quadratic degree bound uses the following entropy estimate for
tournament targets.  All entropies are computed with natural logarithms.

\begin{lemma}\label{lem:entropy-violation}
Let $Q$ be a tournament containing no $\tourn p$, and let
$\mu_1,\ldots,\mu_p$ be probability distributions on $V(Q)$.  Let
$X_i\sim\mu_i$ be independent random variables and put
\[
  \Phi_Q(\mu_1,\ldots,\mu_p)
  :=\sum_{1\le i<j\le p}\Pr(X_i\not\to X_j),
\]
where $X_i=X_j$ counts as nonforward.  Then
\[
  \Phi_Q(\mu_1,\ldots,\mu_p)\ge1
\]
and
\[
  \sum_{i=1}^p H(\mu_i)
  \le (4p-4)(\log2)\Phi_Q(\mu_1,\ldots,\mu_p).
\]
\end{lemma}

\begin{proof}
Every deterministic $p$-tuple in $Q$ has a nonforward pair, since otherwise
its entries are distinct and induce $\tourn p$.  Taking expectations gives
$\Phi_Q\ge1$.

Fix $i$.  For a choice $y=(y_j)_{j\ne i}$, let $W_i(y)$ be the number of
nonforward pairs among its coordinates, and let $A_i(y)$ be the set of all
$x\in V(Q)$ which make every pair involving position $i$ forward.  Thus
\[
  y_j\to x\quad(j<i),
  \qquad
  x\to y_j\quad(j>i).
\]
We claim that
\begin{equation}\label{eq:compatible-colors}
  |A_i(y)|\le2^{W_i(y)}.
\end{equation}
Indeed, form a graph on the positions other than $i$, joining the endpoints
of every nonforward pair.  Taking one endpoint from each edge gives a vertex
cover $S$ of size $s\le W_i(y)$.  Outside $S$, the entries of $y$ form a
forward, and hence transitive, sequence of length $p-1-s$; its entries are
distinct because equality was counted as nonforward.  If $A_i(y)$
contained a transitive subtournament of order $s+1$, insert its vertices at
position $i$ in their transitive order.  Together with the sequence outside
$S$ this would give $\tourn p$ in $Q$.  Hence $A_i(y)$ contains no
$\tourn{s+1}$.  The elementary bound $\ttr(s+1)\le2^s$ now proves
\eqref{eq:compatible-colors}; when $s=0$, it says that $A_i(y)$ is empty.

Let
\[
  I_i:=\sum_{j<i}\Pr(X_j\not\to X_i)
       +\sum_{j>i}\Pr(X_i\not\to X_j).
\]
For fixed $y$, set $a_i(y)=\mu_i(A_i(y))$.  Splitting $\mu_i$ between
$A_i(y)$ and its complement, using $|V(Q)|<\ttr(p)\le2^{p-1}$ and
\eqref{eq:compatible-colors}, gives
\[
  H(\mu_i)
  \le(\log2)\bigl(1+a_i(y)W_i(y)
                  +(p-1)(1-a_i(y))\bigr).
\]
Zero-mass terms in this entropy decomposition are taken to be zero.
Now choose the coordinates $y_j$ independently according to $\mu_j$.  We
have
\[
  \begin{aligned}
    \mathbb E W_i&=\Phi_Q-I_i,\\
    \mathbb E(1-a_i)
      &=\Pr\bigl(\text{some pair involving $i$ is nonforward}\bigr)
        \le I_i.
  \end{aligned}
\]
Since $0\le a_i\le1$, averaging the preceding entropy bound yields
\[
  H(\mu_i)\le(\log2)\bigl(1+\Phi_Q+(p-2)I_i\bigr).
\]
Finally, $\sum_i I_i=2\Phi_Q$.  Summing over $i$ and using
$\Phi_Q\ge1$ gives
\[
  \sum_iH(\mu_i)
  \le(\log2)\bigl(p+(3p-4)\Phi_Q\bigr)
  \le(4p-4)(\log2)\Phi_Q.
\]
\end{proof}

\begin{proof}[Proof of Theorem~\ref{thm:quantitative-lifts}]
Fix a $\tourn p$-free tournament $Q$ of order $q=\ttr(p)-1$.  We use two
random layered graphs.  In both, take disjoint sets
$V_1,\ldots,V_p$, each of order $n$, and independently include each possible
arc from $V_i$ to $V_j$, $i<j$, with probability $\lambda/n$.  The resulting
oriented graph maps to $\tourn p$ by its layer indices.  In each application
$n$ is taken sufficiently large that $n\ge\lambda$.

First set $\lambda=16p\log q$.  For any map
$f\colon\bigcup_iV_i\to V(Q)$, a uniformly random transversal of the layers
has a nonforward image pair.  Double counting transversals shows that at
least $n^2$ potential arcs are bad for $f$.  The number of selected bad arcs
has mean at least $\lambda n$, so the usual binomial lower-tail estimate gives
\[
  \Pr\bigl(f\text{ has fewer than }\lambda n/2
           \text{ selected bad arcs}\bigr)
  \le e^{-\lambda n/8}.
\]
There are $q^{pn}$ maps, and hence the probability that some map has fewer
than $\lambda n/2$ selected bad arcs is at most
\[
  q^{pn}e^{-\lambda n/8}
  =\exp\bigl(pn\log q-2pn\log q\bigr)
  =q^{-pn}.
\]

Put $D=p\lambda$, and let $C_\ell$ be the number of underlying cycles of
length $\ell$.  Counting cyclic orderings gives
\[
  \mathbb E C_\ell
  \le \frac{(pn)^\ell}{2\ell}
       \left(\frac{\lambda}{n}\right)^\ell
  =\frac{D^\ell}{2\ell}.
\]
Since $D\ge2$, the number $Y=\sum_{\ell=3}^h C_\ell$ satisfies
\[
  \mathbb EY
  \le\sum_{\ell=3}^h\frac{D^\ell}{2\ell}
  \le\frac{D^h}{6(1-D^{-1})}
  \le\frac{D^h}{3}.
\]
Take
\begin{equation}\label{eq:order-construction-n}
  n=\left\lceil8D^h/\lambda\right\rceil.
\end{equation}
Then $n\ge\lambda$, and Markov's inequality gives
\[
  \Pr(Y\ge\lambda n/2)
  \le\frac{2D^h}{3\lambda n}\le\frac1{12}.
\]
Together with the failure probability $q^{-pn}<1/4$ above, this gives
an outcome in which every map to $Q$ has at least $\lambda n/2$ bad
arcs and $Y<\lambda n/2$.  Delete one arc from each cycle counted by
$Y$.  Every map retains a bad arc, and the remaining graph has underlying
girth greater than $h$.

By~\eqref{eq:classical-ramsey}, $\lambda=16p\log q=O(p^2)$, so
\[
  pn\le8p^{h+1}\lambda^{h-1}+p=O_h(p^{3h-1}).
\]
As in the proof of Theorem~\ref{thm:high-girth}, a weak component that
does not map to $Q$ belongs to $\mathcal H_g(p)$.
This proves the bound on $m_g(p)$.

For the degree bound, repeat the construction with
\[
  \lambda=24p.
\]
For a map $f$, let $\mu_i$ be the empirical distribution of its values on
$V_i$, and abbreviate
$\Phi=\Phi_Q(\mu_1,\ldots,\mu_p)$.  Exactly $n^2\Phi$ potential arcs are bad
for $f$.  Thus
\begin{equation}\label{eq:type-lower-tail}
  \Pr\bigl(f\text{ has fewer than }\lambda n\Phi/2
           \text{ selected bad arcs}\bigr)
  \le e^{-\lambda n\Phi/8}.
\end{equation}
For a fixed empirical profile $(\mu_1,\ldots,\mu_p)$, the number of maps
having that profile is at most
\[
  \exp\left(n\sum_iH(\mu_i)\right).
\]
Lemma~\ref{lem:entropy-violation} and~\eqref{eq:type-lower-tail} show that the
probability that some map with this profile has fewer than
$\lambda n\Phi/2$ selected bad arcs is at most
\[
  \exp\left(-c_p n\Phi\right),
  \qquad
  c_p:=3p-(4p-4)\log2>0.
\]
There are at most $(n+1)^{pq}$ empirical profiles.  Since $\Phi\ge1$, the
probability that some map violates the desired lower bound is at most
\[
  (n+1)^{pq}e^{-c_pn},
\]
which tends to zero because $p$ and $q$ are fixed while $n$ tends to
infinity.  Hence, with probability tending to one, every map to $Q$ has at
least
\[
  \lambda n\Phi/2\ge\lambda n/2
\]
selected bad arcs.

Set $D=(p-1)\lambda$, and again let $Y$ count the underlying cycles
of length at most $h$.  Since $p/(p-1)\le3/2$, the same calculation gives
\[
  \mathbb EY\le K_hD^h,
  \qquad K_h=\frac13\left(\frac32\right)^h.
\]
If $Z$ is the total degree of a fixed vertex in the original random graph,
then
\[
  Z\sim\operatorname{Bin}\left((p-1)n,\frac{\lambda}{n}\right),
  \qquad \mathbb EZ=D.
\]
The size-biased identity followed by the Chernoff bound gives
\[
  \mathbb E\bigl[Z\mathbf1_{\{Z>4D\}}\bigr]
  =D\Pr\left(
    \operatorname{Bin}\left((p-1)n-1,\frac{\lambda}{n}\right)
       \ge4D\right)
  \le De^{-2D}.
\]
Let $W$ be the number of selected arcs incident with a vertex whose original
degree exceeds $4D$.  Summing the last estimate over all $pn$ vertices gives
\[
  \mathbb EW\le pnDe^{-2D}.
\]
Since $D=(p-1)\lambda=24p(p-1)$, Markov's inequality yields
\[
  \Pr(W>\lambda n/16)
  \le16p(p-1)e^{-2D}<\frac14.
\]
Also, for all sufficiently large $n$,
\[
  \Pr(Y>\lambda n/16)
  \le\frac{16K_hD^h}{\lambda n}<\frac14,
\]
and the probability that some map has fewer than $\lambda n/2$
selected bad arcs is less than $1/4$.  The union bound gives an outcome
satisfying all three properties.  Choose one arc from every cycle counted by $Y$, and delete
these arcs together with all arcs counted by $W$.  At most $\lambda n/8$
arcs are deleted.  The remaining graph has underlying girth greater than $h$,
maximum total degree at most $4D$, and still does not map to $Q$, since every
map originally had at least $\lambda n/2$ bad arcs.

As in the proof of Theorem~\ref{thm:high-girth}, a weak component $C$
that does not map to $Q$ belongs to $\mathcal H_g(p)$.  Since its
maximum total degree is at most $4D$,
\[
  d_g(p)\le4D=96p(p-1).
\]
\end{proof}

The dependence on the girth threshold cannot be polynomial, even when
$p$ is fixed.

\begin{lemma}\label{lem:girth-order-lower}
Let $g\ge4$ be an integer, and let $H$ be an acyclic oriented graph with
$p(H)=4$, $\tau(H)=8$, and $\girth(U(H))>g$.  Then
\[
  |V(H)|\ge3\cdot2^{\lfloor g/4\rfloor}-2.
\]
In particular,
\[
  m_g(4)\ge3\cdot2^{\lfloor g/4\rfloor}-2,
\]
and $m_g(4)=2^{\Theta(g)}$ as $g\to\infty$.  Consequently, $m_g(p)$ has
no upper bound polynomial jointly in $p$ and $g$.
\end{lemma}

\begin{proof}
Since $H\to\tourn4$ and $\tau(H)=8$, there is a $\tourn4$-free
seven-vertex tournament $Q$ such that $H\nrightarrow Q$.  Every vertex
of $Q$ has indegree and outdegree three: four in- or out-neighbors would
contain a transitive triangle and hence yield a copy of $\tourn4$.
Each in- and out-neighborhood therefore induces a cyclic triangle.

Every orientation of a two-edge path maps to $Q$ with arbitrarily
prescribed distinct endpoint images.  To see this, suppose $x\to y$ and
write $\Nplus(x)=\{y,u,v\}$ with $y\to u\to v\to y$.  Then $u$ is a
common out-neighbor of $x,y$, and $x\to v\to y$.  Reversing all arcs
gives a common in-neighbor.  The other two out-neighbors of $y$ lie in
$\Nminus(x)$, so there is also a directed two-arc path from $y$ to $x$.
These are the four possible orientations.

Consequently, every orientation of a three-edge path maps to $Q$ with
arbitrarily prescribed endpoint images $a,b$, which may coincide.  Choose
the image $c$ of the vertex next to the first endpoint in $\Nplus(a)$ or
$\Nminus(a)$, as required by the first edge, with $c\ne b$.  This is
possible because each neighborhood has three vertices.  The remaining
two-edge path then maps to $Q$ with endpoint images $c,b$.

Choose an induced subgraph $H'$ of minimum order such that
$H'\nrightarrow Q$.  The graph $U(H')$ has minimum degree at least two,
since a homomorphism extends over a vertex of degree at most one.  Moreover,
its degree-two vertices are independent.  Otherwise, two adjacent such
vertices, together with their other neighbors, form a three-edge path;
the endpoints are distinct because $\girth(U(H'))>g\ge4$.  Delete the two
internal vertices, map the remaining graph to $Q$ by minimality, and extend
using the preceding claim, a contradiction.

Suppress every degree-two vertex of $U(H')$, replacing its two incident
edges by an edge between its neighbors, and call the resulting graph $K$.
It is nonempty and has minimum degree at least three.  No loop or parallel
edge arises, since $U(H')$ is simple and has no triangle or four-cycle.
Each edge of $K$ corresponds to a path of at most two edges in $U(H')$,
so
\[
  \girth(K)>g/2.
\]
Put $t=\lfloor g/4\rfloor$.  Since $\girth(K)>2t$, a breadth-first
search from any vertex has at least three vertices at distance one and
at least twice as many vertices at each subsequent distance through $t$.
Consequently,
\[
  |V(H)|\ge |V(K)|\ge1+3\sum_{i=0}^{t-1}2^i
  =3\cdot2^t-2.
\]

The classical value $\ttr(4)=8$~\cite[Section~1]{ErdosMoser1964}
shows that this lower bound applies to every member of $\mathcal H_g(4)$.

For the upper bound, take $p=4$ in the first construction of
Theorem~\ref{thm:quantitative-lifts}.  Then $q=7$,
$\lambda=64\log7$, and $D=256\log7$.  With $h=\max\{g,12\}$,
the choice~\eqref{eq:order-construction-n} gives
\[
  m_g(4)\le4n\le128D^{h-1}+4<D^h.
\]
For $g\ge12$, we have $h=g$, so the two bounds give
$m_g(4)=2^{\Theta(g)}$.
\end{proof}

\section{Forbidden oriented trees}\label{sec:forbidden-trees}

An acyclic oriented graph is $F_{0,3}$-free exactly when its maximum
outdegree is at most two.  Thus the pure-claw bound $\tau(H)=O(p(H)^4)$
for nonempty graphs follows from~\cite[Section~3]{GrzesikEtAl2024}; arc reversal gives the
$F_{3,0}$ case.  We now consider mixed claws.

\begin{lemma}\label{lem:claw}
Let $H$ be an acyclic oriented graph with $p=p(H)\ge2$.
If $H$ is $F_{1,2}$-free or $F_{2,1}$-free, then $\tau(H)\le2p-2$.
\end{lemma}

\begin{proof}
Suppose that $H$ is $F_{1,2}$-free, and let $S$ be its set of sources.
Every vertex outside $S$ has an in-neighbor.  If it had two out-neighbors,
these vertices would form a copy of $F_{1,2}$.  Hence
\begin{equation}\label{eq:functional-outdegree}
  \Delta^+(H-S)\le1.
\end{equation}
For each $v\in V(H-S)$, let $d(v)$ be the length of the unique maximal
directed path starting at $v$.  Every arc $uv$ satisfies
$d(u)=d(v)+1$.  Giving $\dcycle3$ vertex set $\mathbb Z_3$ and arcs
$i\to i-1$, the map $v\mapsto d(v)\pmod3$ is a homomorphism
$H-S\to\dcycle3$.  Mapping every source to a new vertex $s$ dominating
the triangle gives
\begin{equation}\label{eq:mixed-claw-target}
  H\longrightarrow T_0:=\{s\}\Rightarrow\dcycle3.
\end{equation}

Every tournament of order $2p-2$ contains $T_0$ or $\tourn p$.  Indeed, if
it contains no $T_0$, the out-neighborhood of every vertex contains no
directed triangle and is therefore transitive.  A vertex of maximum outdegree has at least $p-1$
out-neighbors, and it together with any $p-1$ of them gives $\tourn p$.
Now~\eqref{eq:rank-map} and~\eqref{eq:mixed-claw-target} show that $H$ maps
to every tournament of order $2p-2$.  Arc reversal gives the $F_{2,1}$
case.
\end{proof}

For the directed orientation of $P_4$, exclusion gives $p(H)\le3$,
and hence $\tau(H)\le\ttr(3)\le4$ by~\eqref{eq:universal-upper}
and~\eqref{eq:classical-ramsey}.  For the one-turn orientation, deleting
the sources leaves an out-forest.

\begin{theorem}\label{thm:one-turn}
If $H$ is $\Fturn$-free, acyclic, and $p(H)\ge2$, then
\[
  \tau(H)\le2p(H)-2.
\]
\end{theorem}

\begin{proof}
Let $S$ be the set of sources of $H$ and put $K=H-S$.
Suppose that distinct $x,y\in V(K)$ both send an arc to a vertex $v$.
Choose $x$ before $y$ in a topological ordering of $H$.
Since $x$ is not a source, it has a predecessor $z$.
The vertices $z,x,v,y$ are distinct, and
$z\to x\to v\leftarrow y$ is a copy of $\Fturn$, a contradiction.
Thus every vertex of $K$ has indegree at most one.

Put $q=p(K)$.  The graph $K$ is nonempty because $p(H)\ge2$.
The first vertex of a longest directed path in $K$ has a predecessor
in $H$, which cannot lie on that path by acyclicity.  Hence
$q\le p(H)-1$.  The longest-path ranks in $K$ increase by exactly one
along every arc, since each vertex has at most one predecessor.
Thus $K\to\dpath q$.

Let $T$ be a tournament on $2q$ vertices.  Some vertex $t$ has at least
$q$ out-neighbors, among which there is a directed path of order $q$.
Map $K$ to this path and every vertex of $S$ to $t$.
Since $S$ is independent and every arc incident with $S$ points into $K$,
this gives $H\to T$.  Consequently,
\[
  \tau(H)\le2q\le2p(H)-2.
\]
\end{proof}

When $p(H)=1$, the graph is edgeless and $\tau(H)=1$.

It remains to consider the alternating orientation.  We first treat
triangle-free underlying graphs and then reduce the general case to this
one.  Both arguments use the following lemma.

\begin{lemma}\label{lem:local-rigidity}
Let $H$ be $\Falt$-free, and let $u\to v$ be an arc such that
$\dplus(u)\ge2$ and $\dminus(v)\ge2$.  There is a unique vertex $w$ such
that
\[
  \Nplus(u)=\{v,w\},
  \qquad
  \Nminus(v)=\{u,w\},
  \qquad
  u\to w\to v.
\]
Moreover, $\dminus(w)=\dplus(w)=1$.
\end{lemma}

\begin{proof}
Choose
\[
  x\in\Nplus(u)\setminus\{v\},
  \qquad
  y\in\Nminus(v)\setminus\{u\}.
\]
If $x\ne y$, then $y\to v\leftarrow u\to x$ is a copy of $\Falt$.
Thus both sets equal $\{w\}$, proving the neighborhood equalities and
the arcs $u\to w\to v$.  If $z\to w$ for $z\ne u$, then
$z\to w\leftarrow u\to v$ is a copy of $\Falt$.  Similarly,
$w\to z$ with $z\ne v$ gives $u\to v\leftarrow w\to z$.  Thus $u$ and
$v$ are the unique in- and out-neighbors of $w$.
\end{proof}

If $Q$ is a directed multigraph, its directed line graph $L(Q)$ has one
vertex for each arc of $Q$, with $e\to f$ exactly when the head of $e$ is
the tail of $f$.  Put
\[
  \mathcal S_m=L(\tourn m).
\]
Equivalently, the vertices of $\mathcal S_m$ are the pairs $(i,j)$ with
$1\le i<j\le m$, and $(i,j)\to(j,k)$ whenever $i<j<k$.
For line digraphs, shift graphs, and induced forbidden oriented paths,
see~\cite[Section~4.1]{AboulkerEtAl2018} and the references there.
Here we give the representation under non-induced exclusion and apply
it to compressibility.

\begin{lemma}\label{lem:shift-universal}
For every $n\ge1$ and every tournament $T$ of order $n$, there is a
homomorphism
\[
  \mathcal S_{n+1}\longrightarrow T.
\]
\end{lemma}

\begin{proof}
We argue by induction on $n$.  The assertion is immediate for $n=1$.
Choose $x\in V(T)$ and put
\[
  A=\Nminus_T(x),
  \qquad B=\Nplus_T(x),
  \qquad |A|=a,
  \qquad |B|=b.
\]
Partition $[n+1]$ into consecutive intervals $I,J$ of orders $a+1$ and
$b+1$.  By induction, map the pairs contained in $I$ to $T[A]$ and those
contained in $J$ to $T[B]$; an empty side has no pairs to map.  Map every
crossing pair $(i,j)$, $i\in I$ and $j\in J$, to $x$.

For an arc $(i,j)\to(j,k)$, either all indices lie in one interval or its
two images lie in $A,\{x\}$ in that order, or in $\{x\},B$ in that order.
Since $A\to x\to B$, all arcs are preserved.
\end{proof}

The preceding lemma also gives
\begin{equation}\label{eq:line-compressibility}
  \tau(L(Q))=p(L(Q))=p(Q)-1
\end{equation}
for every acyclic directed multigraph $Q$ with at least one arc.
Indeed, $L(Q)$ is acyclic, and its directed paths correspond to directed
paths in $Q$ with one more vertex; acyclicity prevents repetitions.
Put $m=p(Q)\ge2$.  The longest-path rank map $Q\to\tourn m$ induces
$L(Q)\to\mathcal S_m$, so Lemma~\ref{lem:shift-universal} gives
$\tau(L(Q))\le m-1=p(L(Q))$.  The reverse inequality follows
from~\eqref{eq:universal-lower}.

\begin{theorem}\label{thm:triangle-free-exact}
Let $H$ be a nonempty $\Falt$-free acyclic oriented graph whose underlying
graph is triangle-free.  Then
\[
  \tau(H)=p(H).
\]
\end{theorem}

\begin{proof}
By Lemma~\ref{lem:local-rigidity}, every arc $u\to v$ satisfies
$\dplus(u)=1$ or $\dminus(v)=1$.  Form a bipartite graph $B$ on the
ports $\{v^-:v\in V(H)\}$ and $\{v^+:v\in V(H)\}$, with edge
$u^+v^-$ exactly when $u\to v$.  Each edge has an endpoint of degree
one, so every component is a star or an isolated vertex.

Let $[z]$ denote the component containing a port $z$.  Construct a directed
multigraph $Q$ whose vertices are the components of $B$, with one arc
$e_v:[v^-]\to[v^+]$ for each $v\in V(H)$.  Opposite-side ports in one
component are adjacent, so
\[
  [u^+]=[v^-]
  \quad\Longleftrightarrow\quad u^+v^-\in E(B)
  \quad\Longleftrightarrow\quad u\to v.
\]
Taking $u=v$ shows that $Q$ has no loops, and the same equivalence gives
$H=L(Q)$.  A directed cycle in $Q$ would give one in $H$, so $Q$ is
acyclic.  Since $H$ is nonempty, $Q$ has an arc, and
\eqref{eq:line-compressibility} gives $\tau(H)=p(H)$.
\end{proof}

We now remove the triangle-free assumption.  Call an arc $u\to v$ of an $\Falt$-free graph $H$ \emph{diagonal} if
\[
  \dplus_H(u)\ge2
  \qquad\text{and}\qquad
  \dminus_H(v)\ge2.
\]
By Lemma~\ref{lem:local-rigidity}, every diagonal arc has a unique
\emph{middle vertex} $w$ with $u\to w\to v$.  Let $M=M(H)$ be the set of
all middle vertices.

\begin{lemma}\label{lem:middle-vertices}
Distinct diagonal arcs have distinct middle vertices.  No vertex of $M$ is
an endpoint of a diagonal arc, and $M$ is independent.  If $w\in M$ is the
middle vertex of $u\to v$, then the only arcs incident with $w$ are
$u\to w$ and $w\to v$.
\end{lemma}

\begin{proof}
By Lemma~\ref{lem:local-rigidity}, a middle vertex $w$ has a unique
in-neighbor $u$ and out-neighbor $v$.  These determine its diagonal
$u\to v$, so distinct diagonals have distinct middle vertices.
Since $\dminus(w)=\dplus(w)=1$, $w$ cannot be an endpoint of a diagonal.
Its only neighbors are the endpoints of its own diagonal, which lie
outside $M$.  Thus $M$ is independent.
\end{proof}

\begin{lemma}\label{lem:delete-middle}
Let $H$ be an $\Falt$-free acyclic oriented graph.  Then the graph
$G=H-M(H)$ is $\Falt$-free and $U(G)$ is triangle-free.
\end{lemma}

\begin{proof}
The graph $G$ is $\Falt$-free because it is a subgraph of $H$.
Since $H$ is acyclic, every triangle is transitive.  Its source-to-sink
arc is diagonal, and its remaining vertex lies in $M(H)$ by
Lemma~\ref{lem:local-rigidity}.  Thus deleting $M(H)$ removes every
triangle.
\end{proof}

For a tournament $T$, define its \emph{transition graph} $\Gamma(T)$ on
$V(T)$ as follows: if $x\to y$ in $T$, then $xy$ is an edge of
$\Gamma(T)$ exactly when there is no $z$ with $x\to z\to y$.  An
independent set in $\Gamma(T)$ therefore has the property that every
tournament arc within it can be replaced by a two-arc directed path in $T$.

\begin{lemma}\label{lem:transition-graph}
For every tournament $T$,
\[
  \Delta(\Gamma(T))\le2,
  \qquad
  \alpha(\Gamma(T))\ge
  \left\lceil\frac{|V(T)|}{3}\right\rceil.
\]
\end{lemma}

\begin{proof}
Orient $\Gamma(T)$ as in $T$.  If $x\to y$ and $x\to z$ were two
outgoing arcs, assume $y\to z$.  Then $x\to y\to z$ contradicts
$xz\in E(\Gamma(T))$.  Thus the outdegree is at most one, and reversal
gives the same bound for the indegree.  Hence $\Delta(\Gamma(T))\le2$.
A greedy $3$-coloring has a color class of size at least
$\lceil |V(T)|/3\rceil$.
\end{proof}

\begin{theorem}\label{thm:alternating-general}
Let $H$ be a nonempty $\Falt$-free acyclic oriented graph, let $M=M(H)$,
and put $q=p(H-M)$.  Then
\[
  \tau(H)\le3q-2\le3p(H)-2.
\]
\end{theorem}

\begin{proof}
Since $H$ is nonempty and every middle vertex has neighbors outside
$M$, the graph $G=H-M$ is nonempty.  Lemma~\ref{lem:delete-middle} and
Theorem~\ref{thm:triangle-free-exact} give
\begin{equation}\label{eq:tau-G}
  \tau(G)=q.
\end{equation}
Let $T$ be a tournament of order $3q-2$.  Lemma~\ref{lem:transition-graph}
gives an independent set $S$ of order $q$ in $\Gamma(T)$, and
\eqref{eq:tau-G} gives a homomorphism $\varphi:G\to T[S]$.

For each middle vertex $w$ of a diagonal $u\to v$, both endpoints lie
in $G$.  Since $S$ is independent and $\varphi(u)\to\varphi(v)$, choose
$z_w$ with $\varphi(u)\to z_w\to\varphi(v)$ and set
$\varphi(w)=z_w$.  By Lemma~\ref{lem:middle-vertices}, $M$ is independent
and the only arcs incident with $w$ are $u\to w$ and $w\to v$.
Thus this extends $\varphi$ to a homomorphism $H\to T$.
Hence $\tau(H)\le3q-2$, and $q\le p(H)$
gives the second inequality.
\end{proof}

Together with the claw, directed-path, and one-turn cases proved above,
this completes the proof of Theorem~\ref{thm:four-vertex-trees}.

\section{Concluding remarks}\label{sec:sharpness}

For $q\ge2$, define $H_q$ on
\[
  x_1,\ldots,x_q,
  \qquad
  y_1,\ldots,y_{q-1}
\]
by adding, for $1\le i<q$, the arcs
\begin{equation}\label{eq:Hq-arcs}
  x_i\to x_{i+1},
  \qquad
  x_i\to y_i,
  \qquad
  y_i\to x_{i+1}.
\end{equation}
These triangle chains are the graphs of Wang and
Lu~\cite[Section~4]{WangLu2025} with $k=2q-1$ and
$S=\{1,3,\ldots,2q-3\}$.  Their concluding remarks and the square-path
upper bound recalled there give
\begin{equation}\label{eq:Hq-known}
  \tau(H_q)=3q-2.
\end{equation}
They attain the bound in Theorem~\ref{thm:alternating-general}
expressed in terms of $p(H-M)$.

In a copy of $\Falt$, the central arc has a tail of outdegree at least
two and a head of indegree at least two.  In $H_q$, only the arcs
$x_i\to x_{i+1}$ have both properties.  The two outside vertices would
then both have to be $y_i$, a contradiction.  Thus $H_q$ is $\Falt$-free.

The directed path
\[
  x_1\to y_1\to x_2\to y_2\to\cdots\to y_{q-1}\to x_q
\]
contains every vertex, so $p(H_q)=2q-1$.  Its middle-vertex set is
$M=\{y_1,\ldots,y_{q-1}\}$, and deleting it leaves the directed path
$x_1\to\cdots\to x_q$.  Hence $p(H_q-M)=q$, and
\eqref{eq:Hq-known} gives equality in the bound
$\tau(H)\le3p(H-M)-2$.  In terms of $p(H)$, it gives
\[
  \tau(H_q)=\frac{3p(H_q)-1}{2},
\]
so the coefficient in a general linear bound for $\Falt$-free graphs
must be at least $3/2$.

The construction suggests a refinement of the general alternating bound.
For a nonempty $\Falt$-free acyclic graph $H$, let $M=M(H)$ and $G=H-M$.  Define
$r(H)$ to be the maximum number of diagonal arcs on a directed path of $G$.
If such a path has $k$ vertices and contains $s$ diagonal arcs, replacing
each diagonal by its two-arc path through the corresponding middle vertex
produces a directed path of $H$ with $k+s$ vertices.  Since $k\ge s+1$,
\begin{equation}\label{eq:r-bound}
  r(H)\le\left\lfloor\frac{p(H)-1}{2}\right\rfloor.
\end{equation}

We conjecture that every nonempty $\Falt$-free acyclic oriented graph
$H$ satisfies
\[
  \tau(H)\le p(H)+r(H).
\]
By~\eqref{eq:r-bound}, this would imply
\[
  \tau(H)\le\left\lfloor\frac{3p(H)-1}{2}\right\rfloor.
\]
It contains the triangle-free equality as the case $r(H)=0$ and would be
sharp for $H_q$, where $r(H_q)=q-1$.
Another known equality case follows from Wang and
Lu~\cite[Corollary~3.3]{WangLu2025}.  Start with the directed path
$v_1\to\cdots\to v_k$ and add $v_i\to v_{i+2}$ for $i\in S$, where
distinct elements of $S\subseteq\{1,\ldots,k-2\}$ differ by at least
three.  Their result gives $\tau(H)=k+|S|$.  These graphs are
$\Falt$-free, with $p(H)=k$ and $r(H)=|S|$, so they satisfy
$\tau(H)=p(H)+r(H)$.

It remains to determine which larger oriented forests $F$ make the
$F$-free acyclic class polynomially $\tau$-bounded.  Another open case is
the class of acyclic oriented graphs with maximum outdegree
three~\cite{GrzesikEtAl2024}.

For the graphs in Theorem~\ref{thm:high-girth}, can polynomial order in
$p$ for fixed $g$ and maximum semidegree $O(p^2)$ be attained
simultaneously?  Can the quadratic degree bound be reduced to a linear
bound?  Lemma~\ref{lem:degree-lower} gives the lower bound
$\Omega(p/(\log p)^2)$, while Lemma~\ref{lem:girth-order-lower} rules out
an order bound polynomial jointly in $p$ and $g$.

\end{document}